\documentclass{amsart}%[oneside, 12pt]{amsart}
\usepackage{tikz}
\usepackage{xcolor}
\usepackage{amssymb,latexsym,amsmath,extarrows}
\usepackage{graphicx,mathrsfs,comment}
\usepackage{hyperref,url}

\usepackage{amstext}
\usepackage{bbm}

\numberwithin{equation}{section}

\makeatletter
\newcommand\@avprod[2]{%
  {\sbox0{$\m@th#1\prod$}%
   \vphantom{\usebox0}%
   \ooalign{%
     \hidewidth
     \smash{\vrule height\dimexpr\ht0+1pt\relax depth\dimexpr\dp0+1pt\relax}%
     \hidewidth\cr
     $\m@th#1\prod$\cr
   }%
  }%
}
\newcommand{\avprod}{\mathop{\mathpalette\@avprod\relax}\displaylimits}

\newtheorem{theorem}{Theorem}[section]
\newtheorem{lemma}[theorem]{Lemma}

\newtheorem{proposition}[theorem]{Proposition}

\newtheorem{corollary}[theorem]{Corollary}

\newtheorem{conjecture}[theorem]{Conjecture}

\newcommand{\e}{\epsilon}

\newcommand{\Tau}{\mathcal{T}}

\renewcommand{\e}{\epsilon}
\renewcommand{\b}{\beta}
\renewcommand{\a}{\alpha}

\newcommand{\R}{\mathbb{R}}

\newcommand{\T}{\mathbb{T}  }

\newcommand{\N}{\mathcal{N}}

\newcommand{\C}{\mathbb{C}}

\newcommand{\bP}{\mathbb{P}}
\newcommand{\cP}{\mathcal{P}}
\newcommand{\Z}{\mathbb{Z}}

\begin{document}

\title[]{Small cap decoupling for the paraboloid in $\mathbb{R}^n$}

\author{Larry Guth} \address{Larry Guth\\  Deparment of Mathematics, Massachusetts Institute of Technology, USA}\email{lguth@math.mit.edu}

\author{Dominique Maldague}\address{ Dominique Maldague\\ Deparment of Mathematics, Massachusetts Institute of Technology, USA} \email{dmal@mit.edu}

\author{Changkeun Oh} \address{ Changkeun Oh\\  Deparment of Mathematics, Massachusetts Institute of Technology, USA \\ 
Department of Mathematical Sciences and RIM, Seoul National University, Republic of Korea }\email{changkeun.math@gmail.com}

\maketitle

\begin{abstract}
We extend the small cap decoupling program established by Demeter, Guth, and Wang to paraboloids in $\R^n$, for some range of $p$. %We prove small cap decouplings for the paraboloid for some range of $p$.
\end{abstract}
%\tableofcontents

%\tableofcontents

\section{Introduction}

Let $\bP^{n-1}$ be a truncated paraboloid in $\R^n$;
\begin{equation}
    \bP^{n-1}:=\Big\{(\xi_1,\ldots,\xi_{n-1},\sum_{i=1}^{n-1}\xi_{i}^2): |\xi_i| \leq 1, \; i=1,\ldots,n-1 \Big\}.
\end{equation}
Denote by $\N_{\bP^{n-1}}(R^{-1})$ the $R^{-1}$-neighborhood of the paraboloid $\bP^{n-1}$. Given a parameter $\vec{\alpha}=(\alpha_1,\ldots,\alpha_{n-1}) \in [\frac12,1]^{n-1}$, let $\Gamma_{\vec{\alpha} }(R^{-1})$ be the collection defined by
\begin{equation}\label{0430.12}
\begin{split}
    \Big\{ (B \times \R) &\cap \N_{\bP^{n-1}}(R^{-1}):  
    \\& B= \big( (c_1,\ldots,c_{n-1})+ [0,R^{-\alpha_1}] \times \cdots \times [0,R^{-\alpha_{n-1}}] \big), \, c_i \in R^{-\alpha_i}\Z \Big\}.
\end{split}
\end{equation}
For the parameter $\vec{\alpha}$, we use the notation $|\vec{\alpha}|:=\sum_{i=1}^{n-1}\alpha_i$. By definition, $|\vec{\alpha}| \geq \frac{n-1}{2}$. Note that the cardinality of $\Gamma_{\vec{\alpha}}(R^{-1}) $ is  $R^{|\vec{\alpha}|}$.
For each positive measure set $B \subset \R^n$ and each $F:\R^n \rightarrow \C$ we define
\begin{equation}
    \cP_{B}F(x):= \int_B \widehat{F}(\xi)e(\xi \cdot x)\, d\xi,
\end{equation}
where $e(\cdot)=e^{2\pi i \cdot}$. Our main theorem is as follows.

\begin{theorem}\label{0428.thm11}
Let $n \geq 2$ and $\frac{n-1}{2} \leq |\vec{\alpha}| \leq \frac{n}{2}$.
Then for  $2 \leq p \leq 2+\frac{2}{|\vec{\alpha}| }$ and $\epsilon>0$,
\begin{equation}\label{0526.14}
    \|F\|_{L^p(\R^n)} \leq C_{p,\epsilon} R^{|\vec{\alpha}|(\frac12-\frac1p)+\epsilon} \Big( \sum_{\gamma \in \Gamma_{\vec{\alpha}}(R^{-1}) }\| \mathcal{P}_{\gamma}F \|_{L^p(\R^n)}^p \Big)^{\frac1p}
\end{equation}
for all functions $F: \R^n \rightarrow \C$ whose Fourier transform is supported on $\N_{\bP^{n-1}}(R^{-1})$.
\end{theorem}
For a given $\vec{\a}$, the range of exponents $p$ in Theorem \ref{0428.thm11} is sharp. It is conjectured that Theorem \ref{0428.thm11} holds in the larger range of $|\vec{\a}|\le n-1$. In the case that $n=2$, Theorem \ref{0428.thm11} was proved in \cite{MR4153908}.

One of the original motivations of Demeter, Guth, and Wang for studying small cap decoupling was to prove exponential sum estimates. %We obtain analogous exponential sum estimates. 
%\begin{conjecture}[Conjecture 2.5 from \cite{MR4153908}] For each $n\ge 2$, $0\le \beta\le n-1$, and $s\ge 1$, we have 
%\[\int_{[0,1]^{n-1}\times[0,\frac{1}{N^\beta}]}|\sum_{k=1}^Ne(k_1x_1+\cdots+k_nx^n)|^{2s}dx\lesssim_\e N^\e(N^{s-\beta}+N^{2s-\frac{n(n+1)}{2}}) .\]    
%\end{conjecture}
%Our small cap decoupling for the paraboloid implies an exponential sum estimate. 
The following is the $n$-dimensional analogue of Corollary 3.2 in \cite{MR4153908}. 

\begin{corollary}\label{cor1}
Let $\epsilon>0$, $n\ge 2$, $\frac{n-1}{2}\le |\vec{\alpha}|\le \frac{n}{2}$ and let $2\le p\le 2+\frac{2}{|\vec{\alpha}|}$. The inequality 
\begin{equation*}
    \begin{split}
        \Big( \frac{1}{R^n} \int_{B_R} &\big|\sum_{\substack{B\in\Gamma_{\vec{\a}}(R^{-1})}} a_{B}e\big(x\cdot\xi_B)|^p  \,dx  \Big)^{\frac1p} \leq C_{p,\epsilon} R^{\frac{|\vec{\alpha}|}{2}+\epsilon}
    \end{split}
\end{equation*}
holds true for any collection of vectors $\cup_{B\in\Gamma_{\vec{\a}}(R^{-1})}\{\xi_B\}$ satisfying $\xi_B\in B$, any ball $B_R\subset \R^n$ of diameter $R$, and each $a_B\in\C$ with magnitude $\sim 1$. 
\end{corollary}

We obtain another type of exponential sum estimate which may be compared to Theorem 3.3 of \cite{MR4153908}. For simplicity, we only state the $n=3$ result. 
\begin{corollary}\label{cor2}
Suppose that $\alpha,\beta\in[\frac{1}{2},1]$, %$\frac{1}{2} \leq \alpha,\beta$ 
and $\alpha+\beta \leq \frac32$. Let $H_1$ be interval of length $N^{\b-\a}$ and let $H_2$ be any interval of length $N^{1-2\a}$. %For simplicity, introduce the notation $B:=[0,1] \times [0,N^{\beta-\alpha}]\times [0,N^{2-2\alpha}]$. 

For $2 \leq p \leq 2+\frac{2}{\alpha+\beta}$, and any $a_{n_1,n_2}\in\C$ with $|a_{n_1,n_2}|\sim1$, we have
\begin{equation*}
    \begin{split}
         \int_{[0,1]\times H_1\times H_2} &\big|\sum_{\substack{n_1 \in [0,N^{\alpha}] \cap \Z \\n_2 \in [0,N^{\alpha}] \cap N^{\alpha-\beta}\Z }}a_{n_1,n_2}e\big(n_1x_1+n_2x_2+(n_1^2+n_2^2)x_3  \big) \big|^p  \,dx  \\& \leq C_{p,\epsilon} N^{1-(3-\frac{p}{2})\a+(\frac{p}{2}+1)\b+\e}.
    \end{split}
\end{equation*}
\end{corollary}
It is easy to see that both Corollary \ref{cor1} and Corollary \ref{cor2} are essentially sharp by considering the exponential sums with all coefficients equal to $1$. 
%\begin{corollary}
%Suppose that $\alpha,\beta\in[\frac{1}{2},1]$ %$\frac{1}{2} \leq \alpha,\beta$ 
%and $\alpha+\beta \leq \frac32$. For simplicity, introduce the notation $B:=[0,1] \times [0,N^{\beta-\alpha}]\times [0,N^{2-2\alpha}]$. 

%For $2 \leq p \leq 2+\frac{2}{\alpha+\beta}$, and any complex sequence $\{a_{n_1,n_2}\}$, we have
%\begin{equation*}
%    \begin{split}
%        \Big( \frac{1}{|B|} \int_{B} &\big|\sum_{\substack{n_1 \in [0,N^{\alpha}] \cap \Z \\n_2 \in [0,N^{\alpha}] \cap N^{\alpha-\beta}\Z }}a_{n_1,n_2}e\big(n_1x_1+n_2x_2+(n_1^2+n_2^2)x_3  \big) \big|^p  \,dx  \Big)^{\frac1p}
%        \\& \lesssim N^{(\alpha+\beta)(\frac12-\frac1p)}\big(\sum_{n_1,n_2}|a_{n_1,n_2}|^p \big)^{\frac1p}.
%    \end{split}
%\end{equation*}
%\end{corollary}

Our small cap decoupling for the paraboloid is known to be related to a restriction conjecture for the paraboloid. As far as the authors know, this observation is due to Bourgain (Lemma 6.29 of \cite{MR1097257}). Since this relationship is not explicitly stated there, we state it here. We do not claim that this is new.

Let $d\sigma$ be a surface measure of the truncated paraboloid $\bP^{n-1}$, meaning that for $f\in L^p([0,1]^{n-1})$, we define
\begin{equation}
    \widehat{fd\sigma}(x):=\int_{[0,1]^{n-1}}f(\xi)e(x \cdot (\xi,|\xi|^2) ) \,d\xi.
\end{equation}

\begin{conjecture}[Restriction conjecture, \cite{MR545235}] For $p>2+\frac{2}{n-1}$ we have
\begin{equation}\label{0526.15}
    \|\widehat{f d\sigma}\|_{L^p(\R^n)} \leq C_{p}\|f\|_{L^p([0,1]^{n-1})}.
\end{equation}
    
\end{conjecture}

\begin{proposition}\label{0526.prop14}
Let $n \geq 2$ and $\frac{n-1}{2} \leq |\vec{\alpha}| \leq n-1$.
 If the small cap decoupling inequality
\eqref{0526.14} is true  for  $p=2+\frac{2}{|\vec{\alpha}|}$, then the restriction conjecture \eqref{0526.15} is true  for the same exponent $p$.

Moreover, if the restriction conjecture is true for $p>2+\frac{2}{n-1}$, then the small cap decoupling is true for $p=2+\frac{2}{n-1}$ with $|\vec{\alpha}|=n-1$.
\end{proposition}

We refer to \cite{wang2022improved} and \cite{guo2023dichotomy} for the best known bounds of the restriction conjecture. Our small cap decoupling implies the restriction estimate for $p>2+\frac{4}{n}$, which is obtained by \cite{MR2033842}.

%\todo{Do we want to mention a barrier of improving the range of small cap decoupling?}

%\todo{Do we want to mention our proof strategy and explain obstacle?}

To prove Theorem \ref{0428.thm11}, we repeat the main ideas of \cite{MR4153908}. The first step is a broad/narrow reduction, which we do in \textsection\ref{broad/narrow}. Then in \textsection\ref{mainarg}, we prove a multilinear small cap by (1) interpolating refined canonical decoupling with multilinear restriction and then (2) applying a refined flat decoupling to go from the canonical caps to small caps. The remaining step is to verify that this sequence of inequalities yields the desired right hand side of small cap decoupling. This verification follows from an $L^2$ based high-low frequency analysis which is done in \textsection\ref{kak}. Finally, we verify Proposition \ref{0526.prop14}, which records the restriction implications of our small cap theorem, in \textsection\ref{rest}.

\subsection*{Notation}

For a sequence $\{a_i\}_{i=1}^{n}$, we use the notation 
\begin{equation}
\avprod_{i=1}^{n}{a_i}:=\prod_{i=1}^{n}|a_i|^{\frac1n}.
\end{equation}

Given a ball $B_K(a)$ of radius $K$ centered at $a$, define the weight function adapted to the ball by
\begin{equation}\label{0517.16}
    w_{B_K}(x):=\big(1+\frac{|x-a|}{K}\big)^{-100n}.
\end{equation}

For two non-negative numbers $A_1$ and $A_2$, we write $A_1 \lesssim A_2$ to mean that there exists a constant $C$ such that $A_1 \le C A_2$. We write $A_1 \sim A_2$ if $A_1 \lesssim A_2$ and $A_2 \lesssim A_1$. We also write $A_1 \lesssim_{\epsilon} A_2$ if there exists $C_{\epsilon}$ depending on a parameter $\epsilon$ such that $A_1 \leq C_{\epsilon}A_2$.

For a rectangular box $T$, we denote by $CT$ the dilation of $T$ by a factor of $C$ with respect to its centroid. 

For a function $f:\R^n \rightarrow \C$ and a positive function $w:\R^n \rightarrow \R$, define 
\begin{equation}
    \|f\|_{L^p(w)}= \Big(\int_{\R^n}|f(x)|^p w(x)\,dx\Big)^{\frac1p}.
\end{equation}

\subsection*{Acknowledgements}

We would like to thank Ciprian Demeter, Yuqiu Fu and Shengwen Gan for valuable discussions. LG is supported by a Simons Investigator award. DM is supported by the National Science Foundation under Award No. 2103249. 

\section{Reduction to a multilinear small cap decoupling \label{broad/narrow}}

In this section, we reduce our decoupling problem to a multilinear small cap decoupling inequality. Let us introduce some definitions to state the multilinear decoupling.
\\

Denote by $n(\xi)$ the unit normal vector of $\bP^{n-1}$ at the point $\xi \in \bP^{n-1}$. We say that
$n$ many points $\xi^{(1)}, \ldots, \xi^{(n)} \in \bP^{n-1}$ are $A$-transverse if
\begin{equation}
    n(\xi^{(1)}) \wedge \cdots \wedge n(\xi^{(n)}) \geq A>0.
\end{equation}
Subsets $\tau_1,\ldots,\tau_{n}$ of the set $\N_{\bP^{n-1}}(R^{-1})$ are called $A$-transverse if  
\begin{equation}
    n(\xi^{(1)}) \wedge \cdots \wedge n(\xi^{(n)}) \geq A>0.
\end{equation}
for any points $\xi^{(i)} \in \tau_i \cap \bP^{n-1}$. 

\begin{theorem}[Multilinear small cap decoupling]\label{0428.thm21}
Let $n \geq 2$, $\frac{n-1}{2} \leq |\vec{\alpha}| \leq \frac{n}{2}$, and $A>0$.
Assume that $F_i: \R^n \rightarrow \C$ has the Fourier transform supported on $\N_{\bP^{n-1}}(R^{-1})$ and the Fourier supports of $F_i$ are $A$-transverse. 
Then for  $2 \leq p \leq 2+\frac{2}{|\vec{\alpha}| }$ and $\epsilon>0$
\begin{equation}\label{0429.23}
    \Big\| \avprod_{i=1}^n F_i \Big\|_{L^p(\R^n)} \leq C_{A,p,\epsilon} R^{|\vec{\alpha}|(\frac12-\frac1p)+\epsilon} \avprod_{i=1}^n \Big( \sum_{\gamma \in \Gamma_{\vec{\alpha}}(R^{-1}) }\| \mathcal{P}_{\gamma}F_i \|_{L^p(\R^n)}^p \Big)^{\frac1p}.
\end{equation}
The constant $C_{A,p,\epsilon}$ depends on a choice of $A,p$ and $\epsilon$.
\end{theorem}

\begin{proposition}\label{0428.prop22}

    Theorem \ref{0428.thm11} is true under the assumption of Theorem \ref{0428.thm21}.
\end{proposition}

Let us prove Proposition \ref{0428.prop22}. The proof is a modification of the broad-narrow analysis in \cite{MR2860188}. Since this argument is standard nowadays, we only give a sketch of the proof. Refer to page 139--150 of \cite{MR3971577} for more details.
\\

We introduce some notations. When $\vec{\alpha}=(\frac12, \ldots,\frac12)$, we sometimes use $\Theta(R^{-1})$ for the set $\Gamma_{\vec{\alpha}}(R^{-1})$. Note that this notation is defined for arbitrary number $R>0$.
\\

In the proof of Proposition \ref{0428.prop22}, we may assume that $|\vec{\alpha}|<\frac{n}{2}$. The case $|\vec{\alpha}|=\frac{n}{2}$ follows from H\"{o}lder's inequality. To prove Proposition \ref{0428.prop22}, we use an induction on $R$. Fix a sufficiently large number  $R_0$. We may assume that $\eqref{0526.14}$ is true for $R \leq R_0$ by taking $C_{p,\epsilon}$ sufficiently large (in particular, we assume that $C_{\epsilon} \geq 2CC_{cK^{-(n-1)},p,\epsilon}K^{100n}$ where $K$ will be determined later). By using this as a base of the induction, we may assume that $\eqref{0526.14}$ is true for $\leq R/2$, and it suffices to show \eqref{0526.14} for $\leq R$.
\\

Let us do the broad-narrow anlaysis.
Introduce a significantly large number $K$, which will be determined later. Fix a ball $B_K \subset \R^n$ of the radius $K$. Consider the set
\begin{equation}\label{0524.24}
    \mathcal{C}:= \{ \tau \in \Theta(K^{-1}) : \|F\|_{L^p(B_K)} \leq K^{10n}\|\cP_{\tau}F\|_{L^p(B_K)} \}.
\end{equation}
Recall that the dimension of $\tau \in \Theta(K^{-1})$ is $\sim K^{-\frac12} \times \cdots \times K^{-\frac12} \times K^{-1}$.
There are two scenarios.

\subsubsection*{Case 1. Broad case}

Suppose that there exist sets $\tau_1,\ldots, \tau_{n} \in \mathcal{C}$ such that $\tau$'s are $cK^{-(n-1)}$-transverse for some small number $c$ independent of $K$. Then we have
\begin{equation}\label{0621.25}
    \|F\|_{L^p(B_K)} \leq K^{10n} 
 \avprod_{i=1}^n \|\cP_{\tau_i}F\|_{L^p(B_K)}.
\end{equation}
Let $c(B_K)$ be the center of the ball $B_K$. Note that $\mathcal{P}_{\tau_i}F$ is essentially constant on $B(x_B+v,2)$. For each $i$, choose $v_i$ so that $\|\mathcal{P}_{\tau_i}F\|_{L^{\infty}(B_K)}$ is attained in $B(x_B+v_i,2)$. Define $F_{v_i}$ by
\begin{equation}
    \widehat{F_{v_i}}(\xi):=\widehat{F}(\xi)e(v_i \cdot \xi).
\end{equation}
Then we have
\begin{equation}
    \|\mathcal{P}_{\tau_i}F\|_{L^p(B_K)} \lesssim K^{O(1)}\|\mathcal{P}_{\tau_i}F_{v_i}\|_{L^p(B(x_B,2))}.
\end{equation}
Hence, by the essentially constant property together with \eqref{0621.25}, we obtain 
\begin{equation}
\|F\|_{L^p(B_K)}
     \lesssim K^{50n} 
 \Big\|\avprod_{i=1}^n \cP_{\tau_i}F_{v_i}\Big\|_{L^p(w_{B_{K}})}.
\end{equation}
Here, $w_{B_K}$ is a weight function adapted to the ball $B_K$ (see \eqref{0517.16}).

\subsubsection*{Case 2. Narrow case } Suppose that we are not in the broad case. Then one can see that there must exist a rectangular box $L$ with dimension $\sim K^{-\frac12} \times 1 \times \cdots \times 1$ such that $\bigcup_{\tau \in \mathcal{C}} \tau \subset L$.
For simplicity, assume that 
$L=[0,K^{-\frac12}] \times [0,1]^{n-1}$.
Since $\tau$'s not belonging to $\mathcal{C}$ are essentially an error, we have
\begin{equation}
    \|F\|_{L^p(B_K)} \lesssim \Big\| \sum_{\tau \subset L } \cP_{\tau}F \Big\|_{L^p(B_K)}.
\end{equation}
Note that the Fourier support of $ \sum_{\tau \subset L } \cP_{\tau}F$ is contained in $\mathbb{R} \times \N_{\bP^{n-2}}(CK^{-1})$.
We next apply a decoupling inequality for the paraboloid $\bP^{n-2}$ (Theorem 1.1 of \cite{MR3374964}), and obtain
\begin{equation}
    \|F\|_{L^p({B_K})} \leq C_{dec,\epsilon} K^{\frac{n-2}{2}(\frac12-\frac1p)+\epsilon} \Big( \sum_{\tau \in \Theta(K^{-1}): \tau \subset L  } \|  \cP_{\tau}F \|_{L^p(w_{B_K})}^p \Big)^{\frac1p}.
\end{equation}
Here $w_{B_K}$ is a weight function adapted to the ball $B_K$ (see \eqref{0517.16}).
\\

So we have the estimate
\begin{equation}
\begin{split}
    \|F\|_{L^p(B_K)} &\leq C K^{50n}
    \Big\|\avprod_{i=1}^n \cP_{\tau_i}F\Big\|_{L^p(B_{K})}
    \\& + C_{dec,\epsilon}
    K^{\frac{n-2}{2}(\frac12-\frac1p)+\epsilon} \Big( \sum_{\tau \in \Theta(K^{-1}): \tau \subset L  } \|  \cP_{\tau}F \|_{L^p(w_{B_K})}^p \Big)^{\frac1p}
\end{split}
\end{equation}
for $cK^{-(n-1)}$-transverse sets $\tau_i$ that depend on a choice of the ball $B_K$. Summing over $B_K$ gives
\begin{equation}
\begin{split}
    \|F\|_{L^p(\R^n)} &\leq CK^{100n}
    \Big\|\avprod_{i=1}^n \cP_{\tau_i}F\Big\|_{L^p(\R^n)}
    \\& + 
    C_{dec,\epsilon}K^{\frac{n-2}{2}(\frac12-\frac1p)+\epsilon} \Big( \sum_{\tau \in \Theta(K^{-1})  } \|  \cP_{\tau}F \|_{L^p(\R^n)}^p \Big)^{\frac1p}
\end{split}
\end{equation}
for some $cK^{-(n-1)}$-transverse sets $\tau_i$. We apply Theorem \ref{0428.thm21} to the first term on the right hand side, and obtain
\begin{equation}
\begin{split}
    \|F\|_{L^p(\R^n)} &\leq CC_{cK^{-(n-1)},\epsilon}K^{100n}
     R^{|\vec{\alpha}|(\frac12-\frac1p)+\epsilon}  \Big( \sum_{\gamma \in \Gamma_{\vec{\alpha}}(R^{-1}) }\| \mathcal{P}_{\gamma}F \|_{L^p(\R^n)}^p \Big)^{\frac1p}
    \\& + 
    C_{dec,\epsilon}K^{\frac{n-2}{2}(\frac12-\frac1p)+\epsilon} \Big( \sum_{\tau \in \Theta(K^{-1})} \|  \cP_{\tau}F \|_{L^p(\R^n)}^p \Big)^{\frac1p}.
\end{split}
\end{equation}
Recall that $C_{\epsilon} \geq 2CC_{cK^{-(n-1)},p,\epsilon}K^{100n}$. So the first term is already of the desired form. In order to close the induction, it suffices to prove
\begin{equation}\label{0428.0211}
    \begin{split}
        C_{dec,\epsilon}&K^{\frac{n-2}{2}(\frac12-\frac1p)+\epsilon} \Big( \sum_{\tau \in \Theta(K^{-1})} \|  \cP_{\tau}F \|_{L^p(\R^n)}^p \Big)^{\frac1p}
        \\&
        \leq \frac{1}{2}C_{\epsilon}R^{|\vec{\alpha}|(\frac12-\frac1p)+\epsilon}  \Big( \sum_{\gamma \in \Gamma_{\vec{\alpha}}(R^{-1}) }\| \mathcal{P}_{\gamma}F \|_{L^p(\R^n)}^p \Big)^{\frac1p}.
    \end{split}
\end{equation}
We will apply parabolic rescaling and induction hypothesis (see the paragraph above \eqref{0524.24}) to $\cP_{\tau}F$. This argument is standard (for example, see Proposition 4.1 of \cite{MR3374964}). Let us give some details. For simplicity, let us assume that 
\begin{equation}
\tau = [0,K^{-\frac12}] \times \cdots \times [0,K^{-\frac12}] \times [0,K^{-1}].    
\end{equation}
Define the linear transformation $L$ mapping $\tau$ to $[0,1]^n$, and the function $G$ so that
\begin{equation}
    \widehat{G}(\xi):=\widehat{\cP_{\tau}F}(L^{-1}\xi).
\end{equation}
Note that the Fourier support of $G$ is contained in $\N_{\bP^{n-1}}(KR^{-1})$. Applying the induction hypothesis \eqref{0526.14} to $G$ with the parameter $KR^{-1}$, the frequency is decomposed into  boxes with dimension $(K^{-1}R)^{-\alpha_1} \times \cdots \times (K^{-1}R)^{-\alpha_{n-1}} \times (K^{-1}R)^{-1}$.   Rescaling back, we obtain
\begin{equation}\label{0428.0212}
    \|\cP_{\tau}F\|_{L^p} \leq C_{\epsilon}K^{-|\vec{\alpha}|(\frac12-\frac1p)-\epsilon}
    R^{|\vec{\alpha}|(\frac12-\frac1p)+\epsilon}  \Big( \sum_{\beta  }\| \mathcal{P}_{\beta}F \|_{L^p(\R^n)}^p \Big)^{\frac1p},
\end{equation}
where $\beta$ has dimension 
\begin{equation}
    K^{-\frac12}(K^{-1}R)^{-\alpha_1} \times \cdots \times K^{-\frac12}(K^{-1}R)^{-\alpha_{n-1}} \times R^{-1}.
\end{equation}
Recall that $\gamma \in \Gamma_{\vec{\alpha}}(R^{-1})$ has dimension
\begin{equation}
    R^{-\alpha_1} \times \cdots \times R^{-\alpha_{n-1}} \times R^{-1}.
\end{equation}
Since $\gamma$'s are disjoint, at most $|\beta|/|\gamma| \sim K^{|\vec{\alpha}|-\frac{(n-1)}{2}}$ many $\gamma$ can be contained in the given $\beta$. By a flat decoupling (for example, see Proposition 2.4 of \cite{MR4153908}), we obtain
\begin{equation}\label{0428.0215}
    \|\cP_{\beta}F\|_{L^p(\R^n)}\leq CK^{(|\vec{\alpha}|-\frac{(n-1)}{2})(1-\frac2p)}
    \Big( \sum_{\gamma \in \Gamma_{\vec{\alpha}}(R^{-1}): \gamma \subset \beta  }\| \mathcal{P}_{\gamma}F \|_{L^p(\R^n)}^p \Big)^{\frac1p}.
\end{equation}
Recall that we are trying to prove \eqref{0428.0211}. By combining \eqref{0428.0212} and \eqref{0428.0215}, the term on the left hand side of \eqref{0428.0211} is bounded by
\begin{equation*}
    \begin{split}
        CC_{dec,\epsilon}C_{\epsilon}K^{\frac{n-2}{2}(\frac12-\frac1p)+\epsilon} 
        K^{|\vec{\alpha}|(\frac12-\frac{1}{p})}
        K^{-\frac{(n-1)}{2}(1-\frac2p)}
    R^{|\vec{\alpha}|(\frac12-\frac1p)+\epsilon}\Big( \sum_{\gamma \in \Gamma_{\vec{\alpha}}(R^{-1}) }\| \mathcal{P}_{\gamma}F \|_{L^p(\R^n)}^p \Big)^{\frac1p}.
    \end{split}
\end{equation*}
It suffices to prove
\begin{equation}
    CC_{dec,\epsilon}
    K^{(\frac{n-2}{2}+|\vec{\alpha}|)(\frac12-\frac1p)} 
        K^{-\frac{(n-1)}{2}(1-\frac2p)} \leq \frac12.
\end{equation}
The exponent of $K$ is smaller than zero provided that $|\vec{a}|<\frac{n}{2}$ for any $p > 2$. So we simply take $K$ sufficiently large so that the above inequality is true. This finishes the proof.

\section{Proof of the multilinear decoupling\label{mainarg}}

In this section, we prove the multilinear small cap decoupling (Theorem \ref{0428.thm21}). We follow a two-step decoupling argument in \cite{MR4153908}. There are two ingredients: a refined flat decoupling (Lemma \ref{0429.lem31}) and a refined decoupling (Lemma \ref{0429.ref}).
\medskip

Let us first introduce the refined flat decoupling. Consider the wave packet decomposition. Let $B$ be a rectangular box in $\R^n$. Then we have
\begin{equation}\label{0622.31}
    \cP_BF(x)= \sum_{T \in \mathbb{T}_B }w_T W_T(x).
\end{equation}
Here, $W_T$ is an $L^{\infty}$ normalized smooth approximation of $\chi_T$ with the Fourier support inside a slight enlargement of $B$. The collection $\mathbb{T}_B$ represents a tiling of $\R^n$ with boxes $T$ dual to $2B$. This decomposition is standard, so we omit the details. We refer to Page 1001 of \cite{MR4153908} for the construction and the proof.

\begin{lemma}[Refined flat decoupling, Corollary 4.2 of \cite{MR4153908}]\label{0429.lem31}
Let $B$ be an rectangular box in $\R^n$, and let $B_1,\ldots,B_L$ be a partition of $B$ into rectangular boxes which are translates of each other. Let $\Tau$ be a tiling of $\R^n$ with rectangular boxes $\tau$ dual to $B_i$.

Let $\T_B' \subset \T_B$ be such that $|w_T| \sim w$ for $T \in \T_B'$ and such that each $\tau \in \Tau$ contains either $\sim N$ tubes $T$ or no tubes at all. Then for $2 \leq p < \infty$
\begin{equation}
    \Big\| \sum_{T \in \T_B' }w_TW_T \Big\|_{L^p(\R^n)} \leq C_{p,\epsilon}N^{\epsilon}(\frac{L^2}{N})^{\frac12-\frac1p} \Big(\sum_{i=1}^L \|\cP_{B_i}F\|_{L^p(\R^n)}^p \Big)^{\frac1p}.
\end{equation}

\end{lemma}

Suppose that $\hat{F}$ is supported on $\N_{\bP^{n-1}}(R^{-1})$.
Consider the wave packet decomposition; apply \eqref{0622.31} to the function $\mathcal{P}_{\theta}F$.
\begin{equation}
    F=\sum_{\theta \in \Theta(R^{-1}) }\cP_{\theta}F=\sum_{T \in \T_R(F) } w_TW_T.
\end{equation}
The family $\T_R(F)$ contains the tubes corresponding to all boxes $\theta \in \Theta(R^{-1})$. Note that $T \in \T_R(F)$ has dimension
\begin{equation}
    R^{\frac12} \times \cdots \times R^{\frac12} \times R.
\end{equation}
We also use the notation
\begin{equation}
    F_T:=w_TW_T.
\end{equation}
Here is the refined decoupling formulated in \cite{MR4153908} (see Theorem 5.6 therein or Theorem 4.2 of \cite{MR4055179}).

\begin{lemma}[Refined decoupling]\label{0429.ref}
    Let $\mathcal{Q}$ be a collection of pairwise disjoint squares $q$ in $\R^n$, with side length $R^{\frac12}$. Assume that each $q$ intersects at most $M$ fat tubes $R^{\delta}T$ with $T\in \T_R(F)$ for some $M \geq 1$ and $\delta>0$.

    Then for $2 \leq p \leq 
\frac{2(n+1)}{n-1} $ and $\epsilon>0$
    \begin{equation}
        \|F\|_{L^q(\cup_{q \in \mathcal{Q} }q )} \leq C_{p,\delta,\epsilon}R^{\epsilon} M^{\frac12-\frac1p} \Big( \sum_{T \in \T_R(F) }\|F_T\|_{L^p(\R^n)}^p \Big)^{\frac1p}.
    \end{equation}
\end{lemma}

\subsection{Reduction to incidence estimate}

Let us prove Theorem \ref{0428.thm21}. Recall the inequality \eqref{0429.23}.
\begin{equation}\label{0429.37}
    \Big\| \avprod_{i=1}^n F_i \Big\|_{L^p(\R^n)} \leq C_{A,p,\epsilon} R^{|\vec{\alpha}|(\frac12-\frac1p)+\epsilon} \avprod_{i=1}^n \Big( \sum_{\gamma \in \Gamma_{\vec{\alpha}}(R^{-1}) }\| \mathcal{P}_{\gamma}F_i \|_{L^p(\R^n)}^p \Big)^{\frac1p}.
\end{equation}
This inequality is trivial when $p=2$ by the $L^2$-orthogonality.
It suffices to prove for $p=p_c=2+\frac{2}{|\vec{\alpha}|}$. The general case follows from the interpolation between $p_c$ and $p=2$.
We do the wave packet decomposition to each function $F_i$. We can take $\T_i \subset \T_R(F_i)$ and $\Theta_i(R^{-1}) \subset \Theta(R^{-1})$ so that
\begin{equation}
    F_i=\sum_{\theta \in \Theta_i(R^{-1}) }\cP_{\theta}F =\sum_{T \in \T_i} w_T W_T
\end{equation}
and the directions of $T_1 \in \T_1, \ldots, T_n \in \T_n$ are $A$-transverse in the sense that
\begin{equation}
    n(T_1) \wedge \cdots \wedge n(T_n) \geq A.
\end{equation}
Here $n(T_i)$ is the unit vector parallel to the long direction of the tube $T_i$. By pigeonholing, we may assume that 
\begin{equation}
    |w_{T}| \leq |w_{T'}| \leq 2|w_{T}|
\end{equation}
for all $T,T' \in \T_i$ for each $i$. Since both sides of \eqref{0429.37} are symmetric, we may assume that $|w_{T}|$ is comparable to one for all $T \in \T_i$.

The tube $T \in \T_i$ has dimension $R^{\frac12} \times \cdots \times R^{\frac12} \times R$. For $\gamma \in \Gamma_{\vec{\alpha}}(R^{-1})$, the function $\cP_{\gamma}F_i$ is essentially constant on a slab with dimension
\begin{equation}
    R^{\alpha_1} \times \cdots \times R^{\alpha_{n-1}} \times R
\end{equation}
dual to $\gamma$. We will call such slab $(R^{\vec{\alpha}},R)$-slab $\tau$. For each $\theta \in \Theta_i(R^{-1})$, we take $\gamma \in \Gamma_{\vec{\alpha}}(R^{-1})$ contained in $\theta$, and tile $\R^n$ with $(R^{\vec{\alpha}},R)$-slabs $\tau$ associated with $\gamma$. By pigeonholing, we may assume that the number of tubes $T$ associated with $\theta \in \Theta_i(R^{-1})$ and contained in $\tau$ is either approximately $N_i$ or zero for some number $N_i$. By abusing the notation, we still use the same notation $F_i$ for the function obtained after the pigeonholing. 
Note that the number of $\gamma \in \Gamma_{\vec{\alpha}}(R^{-1})$ contained in the given $\theta \in \Theta_i(R^{-1})$ is comparable to $R^{|\vec{\alpha}|-\frac{n-1}{2}}$.
As a corollary of the refined flat decoupling (Lemma \ref{0429.lem31}), we obtain

\begin{corollary}[cf. Corollary 5.7 of \cite{MR4153908}]\label{0429.cor57} For $2 \leq p \leq \infty$ and any $\epsilon>0$
\begin{equation}
     \|\cP_{\theta}F_i\|_{L^p(\R^n)} \leq C_{p,\epsilon}N_i^{\epsilon} \big(\frac{R^{2(|\vec{\alpha}|-\frac{n-1}{2})}}{N_i} \big)^{\frac12-\frac1p}\Big( \sum_{\gamma \in \Gamma_{\vec{\alpha}}(R^{-1}): \gamma \subset \theta } \|\cP_{\gamma}F_i\|_{L^p(\R^n)}^p \Big)^{\frac1p}.
\end{equation}

\end{corollary}

We do one more pigeonholing on spatial sides. Let $\vec{r}=(r_1,\ldots,r_n) \in \Z^{n}$. For each dyadic number $1 \leq r_i \leq 2R^{1/2}$, denote by $\mathcal{Q}_{\vec{r}}$ be the collection of all dyadic cubes $q$ with side length $R^{\frac12}$ such that whose slight enlargments (more precisely, $Cq$ for some large number $C$) intersect $\sim r_i$ many tubes $T \in \T_i$. Define
\begin{equation}
    S_{\vec{r}}:= \bigcup_{ q \in \mathcal{Q}_{\vec{r}} } q.
\end{equation}
Recall that we are trying to prove \eqref{0429.37}. By pigeonholing on the spatial sides, what we need to prove becomes
\begin{equation}\label{0429.313}
    \Big\| \avprod_{i=1}^n F_i \Big\|_{L^p(S_{\vec{r}}  )} \leq C_{A,p,\epsilon} R^{|\vec{\alpha}|(\frac12-\frac1p)+\frac{\epsilon}{2}} \avprod_{i=1}^n \Big( \sum_{\gamma \in \Gamma_{\vec{\alpha}}(R^{-1}) }\| \mathcal{P}_{\gamma}F_i \|_{L^p(\R^n)}^p \Big)^{\frac1p}.
\end{equation}
We have finished all pigeonholing arguments.

\begin{theorem}[cf. Theorem 5.8 of \cite{MR4153908}]\label{0429.thm34}
For $\frac{2n}{n-1} \leq p \leq \frac{2(n+1)}{n-1}$ and $\epsilon>0$
\begin{gather*}
    \Big\|\avprod_{i=1}^n F_i \Big\|_{L^p(S_{\vec{r}})}
    \\
    \lesssim_{\epsilon}  R^{\epsilon}
    \Big(\avprod_{i=1}^n \frac{r_i^{\frac{n}{n-1}}\# \mathcal{Q}_{\vec{r}} }{ R^{\frac12} \# \mathbb{T}_i }  \Big)^{ (n+1)(\frac1p- \frac{n-1}{2(n+1)} )}
    (\avprod_{i=1}^n r_i )^{\frac{2n}{n-1} (\frac{n-1}{2n}-\frac1p)} \avprod_{i=1}^n \Big( \sum_{\theta \in \Theta_i(R^{-1}) }\|\cP_{\theta}F_{i}\|_{L^p}^p \Big)^{\frac1p}.
\end{gather*}
\end{theorem}

\begin{proof}
By an interpolation argument, it suffices to prove the inequality at $p=\frac{2n}{n-1}$ and $\frac{2(n+1)}{n-1}$. Let us start with $p=\frac{2(n+1)}{n-1}$. We need to prove
\begin{equation}\label{0429.0314}
    \Big\|\avprod_{i=1}^n F_i \Big\|_{L^p(S_{\vec{r}})}
    \lesssim_{\epsilon} R^{\epsilon}
    (\avprod_{i=1}^n r_i )^{\frac{2n}{n-1} (\frac{n-1}{2n}-\frac1p)} \avprod_{i=1}^n \Big( \sum_{\theta}\|\cP_{\theta}F_{i}\|_{L^p}^p \Big)^{\frac1p}
\end{equation}
for $p=\frac{2(n+1)}{n-1}$. By a simple calculation, we can check that the exponent of $\avprod r_i$ is equal to $\frac{1}{n+1}$. We apply Lemma \ref{0429.ref} to $F_i$ with $p=\frac{2(n+1)}{n-1}$, and obtain
\begin{equation}\label{0429.0315}
    \|F_i\|_{L^p(S_{\vec{r}}) } \lesssim_{\epsilon} R^{\epsilon} (r_i)^{\frac12-\frac1p}\Big( \sum_{\theta}\|\cP_{\theta}F_{i}\|_{L^p}^p \Big)^{\frac1p}.
\end{equation}
The exponent of $r_i$ is equal to $\frac{1}{n+1}$. So \eqref{0429.0314} follows from H\"{o}lder's inequality and \eqref{0429.0315}.

We next prove for the point $p=\frac{2n}{n-1}$. By a multilinear restriction estimate on $q$ (Theorem 1.16 and Lemme 2.2 of \cite{MR2275834}),
\begin{equation}
    \Big\| \avprod_{i=1}^n F_i \Big\|_{L^{\frac{2n}{n-1} }(q) } \lesssim_{\epsilon}R^{\epsilon} (R^{\frac12})^{-\frac{1}{2}} \avprod_{i=1}^n \|F_i \|_{L^2(w_q)}.
\end{equation}
By the $L^2$-orthogonality and H\"{o}lder's inequality, we have
\begin{equation}
    \Big\| \avprod_{i=1}^n F_i \Big\|_{L^{\frac{2n}{n-1} }(q) } \lesssim  \avprod_{i=1}^n  \Big\| \big(\sum_{\theta}| \cP_{\theta}F_i|^2 \big)^{\frac12} \Big\|_{L^{\frac{2n}{n-1} }(w_q)}.
\end{equation}
Sum over $q \subset S_{\vec{r}}$. Then the desired bound follows from
\begin{equation*}
    \Big\| \big(\sum_{\theta}| \cP_{\theta}F_i|^2 \big)^{\frac12} \Big\|_{L^{\frac{2n}{n-1} }(\sum_{q \subset S_{\vec{r}}} w_q )}
    \lesssim_{\epsilon} R^{\epsilon} 
    \Big( \frac{r_i^{\frac{n}{n-1}}\# \mathcal{Q}_{\vec{r}} }{ R^{\frac12} \# \mathbb{T}_i }  \Big)^{\frac{n-1}{2n} }
    \Big( \sum_{\theta \in \Theta_i(R^{-1}) }\|\cP_{\theta}F_{i}\|_{L^{\frac{2n}{n-1} }(\R^n)}^{\frac{2n}{n-1} } \Big)^{\frac{n-1}{2n}}.
\end{equation*}
To prove this, we first
 raise $\frac{2n}{n-1}$th power.
\begin{equation}\label{0429.318}
    \Big\| \big(\sum_{\theta}| \cP_{\theta}F_i|^2 \big)^{\frac12} \Big\|_{L^{\frac{2n}{n-1} }(\sum_{q \subset S_{\vec{r}}} w_q )}^{\frac{2n}{n-1}}
    \lesssim_{\epsilon}R^{\epsilon} 
    \Big( \frac{r_i^{\frac{n}{n-1}}\# \mathcal{Q}_{\vec{r}} }{ R^{\frac12} \# \mathbb{T}_i }  \Big)
     \sum_{\theta \in \Theta_i(R^{-1}) }\|\cP_{\theta}F_{i}\|_{L^{\frac{2n}{n-1} }(\R^n)}^{\frac{2n}{n-1} }.
\end{equation}
We will compute both sides and compare them with each other.
Let us first compute the right hand side. Recall that for each $T \in \T_i$, $|w_T|$  is comparable to one. Also the measure of a tube $T$ is comparable to $R^{\frac{n+1}{2}}$. So we have
\begin{equation}
\begin{split}
    \Big( \frac{r_i^{\frac{n}{n-1}}\# \mathcal{Q}_{\vec{r}} }{ R^{\frac12} \# \mathbb{T}_i }  \Big)
     \sum_{\theta \in \Theta_i(R^{-1}) }\|\cP_{\theta}F_{i}\|_{L^{\frac{2n}{n-1} }(\R^n)}^{\frac{2n}{n-1} }
     &\sim 
     \Big( \frac{r_i^{\frac{n}{n-1}}\# \mathcal{Q}_{\vec{r}} }{ R^{\frac12} \# \mathbb{T}_i }  \Big)
     (\# \T_i) R^{\frac{n+1}{2}}
     \\&
     \sim 
     r_i^{\frac{n}{n-1}}(\# \mathcal{Q}_{\vec{r}}) R^{\frac{n}{2}}.
\end{split}
\end{equation}
We next compute the left hand side of \eqref{0429.318}. Each $q$ has measure $R^{\frac{n}{2}}$.  So we have
\begin{equation}
    \Big\| \big(\sum_{\theta}| \cP_{\theta}F_i|^2 \big)^{\frac12} \Big\|_{L^{\frac{2n}{n-1} }(S_{\vec{r}} )}^{\frac{2n}{n-1}}
    \sim \int_{S_{\vec{r}}} \Big( 
    \sum_{T \in \T_i} \chi_T
    \Big)^{\frac{n}{n-1}} \sim r_i^{\frac{n}{n-1}} (\# \mathcal{Q}_{\vec{r}}) R^{\frac{n}{2}}.
\end{equation}
So \eqref{0429.318} is true and this finishes the proof. 
\end{proof}

Let us continue the proof of \eqref{0429.313}.
By combining Theorem \ref{0429.thm34} and Corollary \ref{0429.cor57}, we obtain
    \begin{align*}
         \Big\|\avprod_{i=1}^n F_i \Big\|_{L^p(S_{\vec{r}})}
    \lesssim_{\epsilon}  R^{\epsilon}
    \Big(\avprod_{i=1}^n &\frac{r_i^{\frac{n}{n-1}}\# \mathcal{Q}_{\vec{r}} }{ R^{\frac12} \# \mathbb{T}_i }  \Big)^{ (n+1)(\frac1p- \frac{n-1}{2(n+1)} )}
    (\avprod_{i=1}^n r_i )^{\frac{2n}{n-1} (\frac{n-1}{2n}-\frac1p)} 
    \\&  \times
    \avprod_{i=1}^n \big(\frac{R^{2(|\vec{\alpha}|-\frac{n-1}{2})}}{N_i} \big)^{\frac12-\frac1p}\Big( \sum_{\gamma \in \Gamma_{\vec{\alpha}}(R^{-1}) } \|\cP_{\gamma}F_i\|_{L^p(\R^n)}^p \Big)^{\frac1p}.
    \end{align*}
   By the above inequality, \eqref{0429.313} follows from
\begin{equation*}
    \Big(\avprod_{i=1}^n \frac{r_i^{\frac{n}{n-1}}\# \mathcal{Q}_{\vec{r}} }{ R^{\frac12} \# \mathbb{T}_i }  \Big)^{ (n+1)(\frac1p- \frac{n-1}{2(n+1)} )}
    (\avprod_{i=1}^n r_i )^{\frac{2n}{n-1} (\frac{n-1}{2n}-\frac1p)} 
    \avprod_{i=1}^n \big(\frac{R^{2(|\vec{\alpha}|-\frac{n-1}{2})}}{N_i} \big)^{\frac12-\frac1p} \lesssim_{\epsilon} R^{|\vec{\alpha}|(\frac12-\frac1p)+\epsilon}. 
\end{equation*}
Since $p=p_c=2+\frac{2}{|\vec{\alpha}|}$, after some routine computations, this becomes
\begin{equation}
    \# \mathcal{Q}_{\vec{r}} \lesssim_{\epsilon} R^{\epsilon}\big( \avprod_{i=1}^n \# \T_i \big)  \Big(\avprod_{i=1}^n \frac{R^{\frac{n-1}{2}} N_i}{(r_i)^{2|\vec{\alpha}|-n+2} } \Big)^{\frac{1}{2|\vec{\alpha}|-n+1}}.
\end{equation}

This follows from
\begin{proposition}[cf. Proposition 5.9 of \cite{MR4153908}]\label{0429.prop35}
For any $i$ and $\frac{n-1}{2} < |\vec{\alpha}| \leq \frac{n}{2}$
    \begin{equation}
    \# \mathcal{Q}_{\vec{r}} \lesssim_{\epsilon} R^{\epsilon}  ( \# \T_i )  \Big(\frac{R^{\frac{n-1}{2}} N_i}{(r_i)^{2|\vec{\alpha}|-n+2} } \Big)^{\frac{1}{2|\vec{\alpha}|-n+1}}.
\end{equation}
\end{proposition}
We will give a proof of this proposition in the next subsection.

\subsection{Proof of the incidence estimate \label{kak}}

Let us prove Proposition \ref{0429.prop35}. We use a high-low method in \cite{MR4034922}.
Recall that $\mathcal{Q}_{\vec{r}}$ is the collection of all dyadic cubes $q$ with side length $R^{\frac12}$ whose slight enlargements intersect $\sim r_i$ many tubes $T \in \T_i$.
For each $T \in \T_i$, we let $v_T$ be a smooth approximation of $\chi_T$ with the Fourier transform supported on the dual box of $T$ through the origin. Define
\begin{equation}
    K(x):=\sum_{T \in \T_i}v_T(x).
\end{equation}
Define the set
\begin{equation}
    U_{\vec{r}}:= \{ x\in \R^n: K(x) \sim r_i \}.
\end{equation}
Since $q$ has measure $\sim R^{\frac{n}{2}}$, Proposition \ref{0429.prop35} follows from
\begin{equation}\label{0430.325}
    |U_{\vec{r}}| \lesssim_{\epsilon} R^{\epsilon} R^{\frac n2}( \# \T_i )  \Big(\frac{R^{\frac{n-1}{2}} N_i}{(r_i)^{2|\vec{\alpha}|-n+2} } \Big)^{\frac{1}{2|\vec{\alpha}|-n+1}},
\end{equation}
where $|U_{\vec{r}}|$ denotes the measure of the set $U_{\vec{r}}$.
Without loss of generality, we may assume that
\begin{equation}
    \alpha_1 = \max_{1 \leq j \leq n-1}\alpha_j.
\end{equation}
Recall that $\alpha_j \geq \frac12$.
Define the set
\begin{equation}\label{11.15.328}
    B(\alpha):=[-R^{-\alpha_1},R^{-\alpha_1}] \times \cdots \times [-R^{-\alpha_{n-1}},R^{-\alpha_{n-1}}] \times [-R^{-\frac12},R^{-\frac12}].
\end{equation}
For any positive real number $a$, $aB(\alpha)$ is a dilation of the set $B(\alpha)$ by the scale $a$.
Suppose that $(t_0)^{-1}$ is the largest dyadic number smaller than or equal to $R^{\alpha_1}$.
For dyadic numbers $R^{\frac12} \leq t^{-1} \leq R^{\alpha_1}$, introduce Schwarz functions $\eta_t$ so that
\begin{itemize}
\item $\hat{\eta}_{t_0}$ is a smooth bump function on $B(\alpha)$
    \item $\hat{\eta}_t$ is a smooth bump function on $(\frac{2t}{t_0})B(\alpha) \setminus (\frac{t}{t_0})B(\alpha) $ for $t>t_0$
    \item ${\sum_{t \geq t_0}} \hat{\eta}_t(\xi)=1$ for all $\xi \in B_{R^{-\frac12}}$ 
\end{itemize} 

This decomposition looks involved. Let us give explanation of it. The way we decompose frequencies is a dyadic decomposition by using a dilation of the set $\eqref{11.15.328}$. This set involves the parameters $\alpha_1,\ldots,\alpha_{n-1}$. This is related to the fact that on the right hand side of \eqref{0526.14} frequencies are decomposed into $\gamma \in \Gamma_{\vec{\alpha}}(R^{-1})$ (see \eqref{0430.12} for the definition of $\Gamma_{\vec{\alpha}}(R^{-1})$). For simplicity, let us consider the special case that $n=2$ and compare our decomposition with that of \cite{MR4153908}. These two decompositions are basically the same, but there are some differences. In \cite{MR4153908}, a dyadic decomposition using a dilation of the box $\tilde{B}(\alpha):=[-R^{-1},R^{-1}]^2$ is used. More precisely, the following decomposition is used.
\begin{equation}
\begin{split}
    &[-R^{-1/2},R^{1/2}]^2
    \\&= \Big(\bigcup_{j=0}^{\frac12\log_2 R} [-2^{j+1}R^{-1},2^{j+1}R^{-1}]^2 \setminus [-2^jR^{-1},2^jR^{-1}]^2\Big) \cup [-R^{-1},R^{-1}].
    \end{split}
\end{equation}
On the other hand, we used a different box $B(\alpha)=[-R^{-\alpha_1},R^{-\alpha}] \times [-R^{-\frac12},R^{-\frac12}]$. 
One main difference between these two decompositions is that our decomposition uses the smallest final scale $R^{-\alpha_1}$ (the smallest scale in \cite{MR4153908} is $R^{-1}$). Our decomposition simplifies the argument in \cite{MR4153908}. For the case $n>2$, it looks also possible to prove Proposition \ref{0429.prop35} with a dyadic decomposition using a dilation of the box $[-R^{-1},R^{-1}]^n$. The decomposition itself looks simple. However, since our decoupling inequality involves many parameters $\alpha_1,\ldots,\alpha_{n-1}$, the computations and notations get very involved later in the proof. Our decomposition might look involved at first, but in the proof, it gives a cleaner and simpler argument. This is the reason why we use the set \eqref{11.15.328}.
\\

We next define the sets
\begin{equation}
\begin{split}
   & \Omega_t:=\{ x \in \R^n:   |K * \eta_t (x)| \lesssim 
(\log R) r_i   \},
    \\&L:=\{ x \in \R^n: |K*\eta_{t_0}(x)| \lesssim (\log R) r_i \}.
\end{split}
\end{equation}
Note that
\begin{equation}\label{0624.330}
    U_{\vec{r}} \subset L \cup \big(\cup_{R^{-\alpha_1} \leq t \leq R^{-\frac12}} \Omega_t\big).
\end{equation}

Let us first show that
\begin{equation}\label{0430.328}
    |\Omega_t| \lesssim R^{\frac n2}( \# \T_i )  \Big(\frac{R^{\frac{n-1}{2}} N_i}{(r_i)^{2|\vec{\alpha}|-n+2} } \Big)^{\frac{1}{2|\vec{\alpha}|-n+1}}.
\end{equation}
Here is a geometric lemma.
\begin{lemma}\label{0430.lem36}
Let $R^{-1} \leq t \leq R^{-\frac12}$.
For every $\xi \in \R^n$ with $t \leq |\xi| \leq 2t$
\begin{equation}
    \#\{ \theta \in \Theta(R^{-1}): \xi \in (\theta-\theta)  \} \sim R^{\frac{n-3}{2}}t^{-1}.
\end{equation}
\end{lemma}

\begin{proof}
     $\theta-\theta$ is a rectangular box centered at the origin with dimension 
    \begin{equation}
     R^{-\frac12} \times \cdots \times R^{-\frac12} \times R^{-1}.
    \end{equation}
    Hence, we have $|(\theta-\theta) \cap (B_t \setminus B_{t/2})| \sim R^{-1}t^{n-1}$. Recall that $\# \Theta(R^{-1}) \sim R^{\frac{n-1}{2}}$. By symmetry, for given $\xi$, the number of $\theta$ satisfying $\xi \in \theta-\theta$ is comparable to
    \begin{equation}
        \frac{\# \Theta(R^{-1}) \cdot |(\theta-\theta) \cap (B_t \setminus B_{t/2})| }{ |B_{t} \setminus B_{t/2}| } \sim \frac{R^{\frac{n-1}{2}} R^{-1}t^{n-1} }{ t^n } \sim R^{\frac{n-3}{2}}t^{-1}.
    \end{equation}
    This completes the proof.
\end{proof}

Let us continue the proof. We claim that for $R^{-\alpha_1} \leq t \leq R^{-\frac12}$ 
\begin{equation}\label{0430.332}
    \| K*\eta_t \|_{L^2(\R^n)}^2 \lesssim N_i (\# \T_i) R^{n-\frac12}.
\end{equation}
Recall that $K=\sum_T v_T$.
By Plancherel's theorem
\begin{equation}\label{10.10.337}
\begin{split}
    \|K*\eta_t\|_{L^2}^2 = \|\hat{K}\hat{\eta}_t\|_{L^2}^2. 
\end{split}
\end{equation}
The support of $\hat{\eta}_t$ is contained in
\begin{equation}\label{0624.336}
    (\frac{2t}{t_0})B(\alpha) \setminus (\frac{t}{t_0})B(\alpha) \subset  \bigcup_{t/2 \leq \tilde{t}} \{ \xi: \tilde{t} \leq |\xi| \leq 2\tilde{t}\}.
\end{equation}
Hence, by Lemma \ref{0430.lem36}, \eqref{10.10.337} is bounded by
 \begin{equation}
     \lesssim R^{\frac{n-3}{2}}t^{-1} \sum_{\theta}\|\sum_{T \in \T_{\theta}} \hat{v}_T \hat{\eta}_t\|_2^2 \sim R^{\frac{n-3}{2}}t^{-1} \sum_{\theta} \|\sum_{T \in \T_{\theta}} v_T * \eta_t\|_{L^2}^2.
 \end{equation}
 Note that the Fourier support of the function $\sum_{T \in \T_{\theta} }v_T * \eta_t$ is contained in
 \begin{equation}\label{0624.339}
     (\theta -\theta) \cap \frac{2t}{t_0}B(\alpha).
 \end{equation}
The volume of this set plays a role later in the proof. Let us record it as a lemma.
\begin{lemma}\label{lem37} We have
    \begin{equation}
        \Big| (\theta -\theta) \cap \frac{2t}{t_0}B(\alpha) \Big| \sim  \frac1R \prod_{i=1}^{n-1} \min{\big(R^{-\frac12},\frac{t}{t_0}R^{\alpha_1-\alpha_i}\big) }.
    \end{equation}
In particular,
\begin{equation}
    \Big| (\theta -\theta) \cap \frac{2t}{t_0}B(\alpha) \Big| \lesssim R^{-\frac{n+1}{2}}.
\end{equation}
\end{lemma}
\begin{proof}
    Note that  the set $\theta -\theta$ is contained in a box with dimension $\sim R^{-1} \times R^{-1/2} \times \cdots \times R^{-1/2}$. The angle between the  direction of the short length of the box and the vector $(0,\ldots,0,1) \in \mathbb{R}^n$ is $\gtrsim 1$. By this property, 
    \begin{equation}
        \Big| (\theta -\theta) \cap \frac{2t}{t_0}B(\alpha) \Big| \sim  \min{(R^{-\frac12},R^{-1})} \cdot \prod_{i=1}^{n-1} \min{\big(R^{-\frac12},\frac{t}{t_0}R^{\alpha_1-\alpha_i}\big) }.
    \end{equation}
    This gives the desired bound.
\end{proof}
 
 The set \eqref{0624.339} itself is not a rectangular box. By this technical issue, we take a rectangular box contained in \eqref{0624.339} but its 100-enlargement contains \eqref{0624.339}. Denote by $\tau_t(\theta)$ a translated copy of a dual box of the set.
Now we have
\begin{equation}\label{11.14.343}
    \|K*\eta_t\|_2^2 \lesssim 
    R^{\frac{n-3}{2}}t^{-1} \sum_{\theta}\sum_{\tau_t(\theta)   } \|\sum_{T \in \T_{\theta}} v_T * \eta_t\|_{L^2(\tau_t(\theta))}^2.
\end{equation}
To bound the right hand side, we introduce some notations. Denote by $\T_{\theta}$ a collection of $T \in \T_i$  associated with $\theta$. We also define
\begin{equation}
    \T_{\theta} \cap \tau_t:= \{ T \in \mathbb{T}_{\theta}: T \subset 100\tau_t \}.
\end{equation}
Note that since $\tau_t$ has smaller dimensions than an $(R^{\alpha},R)$-slab, we have the bound
\begin{equation}
    \# (\T_{\theta} \cap \tau_t) \lesssim N_i.
\end{equation}
for any $\theta$ and $\tau_t$.
The right hand side of \eqref{11.14.343} is bounded by
\begin{equation}
    \begin{split}
        &\sim 
        R^{\frac{n-3}{2}}t^{-1} \sum_{\theta}\sum_{\tau_t(\theta)   } \Big( \frac{1}{|\tau_t|} \int_{\tau_t}\sum_{T \in \T_{\theta}}v_T(x)\,dx \Big)^2 |\tau_t|
        \\&
        \sim
        R^{\frac{n-3}{2}}t^{-1} \sum_{\theta}\sum_{\tau_t(\theta)   } \Big( (\# \T_{\theta} \cap \tau_t) \cdot |T| \Big)^2 |\tau_t|^{-1}
        \\&
        \lesssim 
        R^{\frac{n-3}{2}}t^{-1}
        \big(\sup_{\theta, \tau_t} (\# \T_{\theta} \cap \tau_t) \big) \Big( \sum_{\theta}\sum_{\tau_t(\theta)} (\#\T_{\theta} \cap \tau_t) \Big) |T|^2 |\tau_t|^{-1}
        \\&
        \lesssim  R^{\frac{n-3}{2}}t^{-1}
        N_i (\# \T_i)R^{n+1}|\tau_t|^{-1}.
    \end{split}
\end{equation}
By Lemma \ref{lem37}, for $R^{-\alpha_1} \leq t \leq R^{-\frac12}$
\begin{equation}
    t^{-1}|\tau_t|^{-1} \lesssim R^{\frac12}R^{-\frac{n+1}{2}} \sim R^{-\frac n2}.
\end{equation}
So we have proved \eqref{0430.332}. 

By the definition of $\Omega_t$, we have
\begin{equation}
    r_i^2 |\Omega_t| \lesssim \|K*\eta_t\|_{L^2}^2 \lesssim N_i (\# \T_i) R^{n-\frac12}.
\end{equation}
By this inequality and routine computations, \eqref{0430.328} follows from
\begin{equation}
    (r_i)^{\frac{n-2|\vec{\alpha}|}{2|\vec{\alpha}|-n+1}} \lesssim  \big(R^{\frac{n-1}{2}}N_i \big)^{\frac{n-2|\vec{\alpha}|}{2|\vec{\alpha}|-n+1}}.
\end{equation}
By the trivial inequality $r_i \lesssim R^{\frac{n-1}{2}}N_i$, the above inequality follows from
\begin{equation}
    \frac{n-2|\vec{\alpha}|}{2|\vec{\alpha}|-n+1} \geq 0.
\end{equation}
This is equivalent to
\begin{equation}
    \frac{n-1}{2} <|\vec{\alpha}| \leq \frac{n}{2}.
\end{equation}
This finishes the proof of \eqref{0430.328}.
\\

It remains to prove
\begin{equation}\label{0430.345}
    |L| \lesssim R^{\frac n2}( \# \T_i )  \Big(\frac{R^{\frac{n-1}{2}} N_i}{(r_i)^{2|\vec{\alpha}|-n+2} } \Big)^{\frac{1}{2|\vec{\alpha}|-n+1}}.
\end{equation}
The desired bound \eqref{0430.325} follows from this inequality, \eqref{0430.328}, and \eqref{0624.330}.
The above inequality will follow from $L^1$-estimate. First, we have
\begin{equation}
    r_i|L| \sim \int_{L} \sum_{T \in \T_i } v_T*\eta_{t_0}(x)\,dx \lesssim (\# \T_i)R^{\frac{n+1}{2}}.
\end{equation}
By this inequality,
\eqref{0430.345} follows from
\begin{equation}
    (r_i)^{\frac{1}{2|\vec{\alpha}|-n+1}} \lesssim R^{-\frac12} (R^{\frac{n-1}{2}}N_i)^{\frac{1}{2|\vec{\alpha}|-n+1}}.
\end{equation}
By raising $(2|\vec{\alpha}|-n+1)$th power on both sides, this becomes
\begin{equation}\label{0624.350}
    r_i \lesssim N_i R^{-|\vec{\alpha}|+n-1}.
\end{equation}
Let us prove this.
For $x \in L$,
\begin{equation}\label{0430.346}
    r_i \sim K* \eta_{t_0}(x) \sim \sum_{\theta}\frac{1}{|\tau_{t_0}|} \int_{\tau_{t_0}}\sum_{T \in \T_{\theta}} v_T(y)\,dy.
\end{equation}
Here $\tau_{t_0}$ has dimension 
\begin{equation}
    R^{\alpha_1} \times \cdots \times R^{\alpha_{n-1}} \times R
\end{equation}
and contains the point $x$. By the definition of $N_i$ (see the discussion above Corollary \ref{0429.cor57}) and \eqref{0430.346}, 
\begin{equation}
    r_i \lesssim (\# \theta)R^{-|\vec{\alpha}|-1}N_iR^{\frac{n+1}{2}} \lesssim N_iR^{-|\vec{\alpha}|+n-1}.
\end{equation}
This completes the proof of \eqref{0624.350}, thus the proof of \eqref{0430.345}.

\section{Proof of Proposition \ref{0526.prop14}\label{rest}}

Let us prove Proposition \ref{0526.prop14}.  Let us assume  the small cap decoupling inequality
\eqref{0526.14}  for  $p=2+\frac{2}{|\vec{\alpha}|}$. We will show the restriction conjecture \eqref{0526.15} for the same exponent $p$. 

By Tao's $\epsilon$-removal lemma \cite{MR1666558}, it suffices to prove the localized version of the restriction estimate:
\begin{equation}\label{0526.41}
    \|\widehat{f d\sigma}\|_{L^p(B_R)} \leq C_{p,\epsilon}R^{\epsilon}\|f\|_{L^p([0,1]^{n-1})}.
\end{equation}

\begin{proposition}[cf. Proposition 4.3 of \cite{MR1625056}]\label{0526.prop41}
    The following is true
    \begin{equation}\label{05260.42}
         \|F\|_{L^p(B_R)} \leq CC_{p,\epsilon}R^{-(1-\frac1p)+\epsilon}\|\widehat{F}\|_{L^p(\R^n)}
    \end{equation}
    for all functions $F: \R^n \rightarrow \C$ whose Fourier transform is supported on $\N_{\bP^{n-1}}(R^{-1})$ if and only if
 \eqref{0526.41} is true for all $f$.
\end{proposition}

The proof of Proposition \eqref{0526.prop41} is identical to that for Proposition 4.3 of \cite{MR1625056}. We leave out the details.
\medskip

By the proposition, our goal will be to prove \eqref{05260.42}. Note that \eqref{05260.42} follows from
\begin{equation}\label{05260.43}
         \|F\|_{L^p(\R^n
         )} \leq CC_{p,\epsilon}R^{-(1-\frac1p)+\epsilon}\|\widehat{F}\|_{L^p(\R^n)}
    \end{equation}
by replacing $F$ with $F\psi_{B_R}$ where $\widehat{\psi_{B_R}}$ is supported on the ball of radius $R^{-1}$ centered at the origin. This argument is standard (for example, see the proof of Proposition 4.3 of \cite{MR1625056}), so we omit the details.

Recall that we 
 are assuming the following inequality (small cap decoupling).
\begin{equation}\label{2023.41}
    \|F\|_{L^p(\R^n)} \leq C_{p,\epsilon} R^{|\vec{\alpha}|(\frac12-\frac1p)+\epsilon} \Big( \sum_{\gamma \in \Gamma_{\vec{\alpha}}(R^{-1}) }\| \mathcal{P}_{\gamma}F \|_{L^p(\R^n)}^p \Big)^{\frac1p}.
\end{equation}
Since the left hand sides of \eqref{2023.41} and \eqref{05260.43} are the same, it suffices to prove
\begin{equation}
    \Big( \sum_{\gamma \in \Gamma_{\vec{\alpha}}(R^{-1}) }\| \mathcal{P}_{\gamma}F \|_{L^p(\R^n)}^p \Big)^{\frac1p} \lesssim R^{-|\vec{\alpha}|(\frac12-\frac1p)}R^{-(1-\frac1p)+\epsilon}\|\widehat{F}\|_{L^p}.
\end{equation}
Note that $\sum_{\gamma}\|\widehat{\mathcal{P}_{\gamma}F}\|_{L^p}^p \lesssim \|\widehat{F}\|_{L^p}^p $. So the above inequality follows from
\begin{equation}\label{0526.44}
    \|{\mathcal{P}_{\gamma}F}\|_{L^p} \lesssim R^{-|\vec{\alpha}|(\frac12-\frac1p)}R^{-(1-\frac1p)+\epsilon}\|\widehat{\mathcal{P}_{\gamma}F}\|_{L^p}.
\end{equation}
To prove this inequality, we first note two trivial inequalities.
\begin{equation}
    \begin{split}
        &\|{\mathcal{P}_{\gamma}F}\|_{L^2} = \|\widehat{\mathcal{P}_{\gamma}F}\|_{L^2},
        \\&
        \|\mathcal{P}_{\gamma}F\|_{L^{\infty}} \lesssim \|\widehat{\mathcal{P}_{\gamma}F}\|_{L^1} \lesssim |\gamma|^{\frac12}\|\widehat{\mathcal{P}_{\gamma}F}\|_{L^2}.
    \end{split}
\end{equation}
By H\"{o}lder's inequality,
\begin{equation}
\begin{split}
    \|{\mathcal{P}_{\gamma}F}\|_{L^p} &\lesssim |\gamma|^{\frac12-\frac1p}\|\widehat{\mathcal{P}_{\gamma}F}\|_{L^2}
    \\&
    \lesssim |\gamma|^{2(\frac12-\frac1p)}\|\widehat{\mathcal{P}_{\gamma}F}\|_{L^p}.
\end{split}
\end{equation}
Hence, \eqref{0526.44} follows from
\begin{equation}
    |\gamma|^{2(\frac12-\frac1p)} \lesssim R^{-|\vec{\alpha}|(\frac12-\frac1p)}R^{-(1-\frac1p)}.
\end{equation}
Recall that $|\gamma|=R^{-|\vec{\alpha}|-1}$ and $p=2+\frac{2}{|\vec{\alpha}|}$. By routine computations, one can check that the above inequality is true for $|\vec{\alpha}|>0$. This finishes the proof.
\\

We next assume the restriction conjecture and show that the small cap decoupling for $p>2+\frac{2}{n-1}$ with $|\vec{\alpha}|=n-1$. 

Fix $\epsilon>0$. Take $p$ sufficiently close to $p_c:=2+\frac{2}{n-1}$ but larger than the number. Let $\vec{\alpha}=(1,\ldots,1)$. Take a smooth bump function $\psi_{B_R}$ of a ball $B_R$ such that the Fourier support of the function is in $B_{R^{-1}}(0)$. By the restriction conjecture and Proposition \ref{0526.prop41}, we have
\begin{equation}\label{0629.411}
\begin{split}
    \|F\|_{L^{p_c}(B_R)} &\lesssim R^{\epsilon} \|\psi_{B_R}F\|_{L^{p}(B_R)} \lesssim R^{-(1-\frac1p)+2\epsilon}\|\widehat{\psi_{B_R}}* \widehat{F}\|_{L^p(\R^n)}.
\end{split}
\end{equation}
By the definition of $\mathcal{P}_{\gamma}$, we have
\begin{equation}
    \|\widehat{F}\|_{L^p(\R^n)}=\Big( \sum_{\gamma \in \Gamma_{\vec{\alpha}}(R^{-1}) }\| \widehat{\mathcal{P}_{\gamma}F }\|_{L^p(\R^n)}^p \Big)^{\frac1p}.
\end{equation}
By some technical issue, this equality makes some trouble in the later step. We will use a smooth partition of unity. For each $\gamma \in \Gamma_{\vec{\alpha}}(R^{-1})$, take a function $\Xi_{\gamma}$ so that 
\begin{gather}
     0 \leq \widehat{\Xi_{\gamma}}(\xi) \textit{ on  } \xi \in \R^n
    \\ \widehat{\Xi_{\gamma}}(\xi) \sim 1 \textit{ on  } \xi \in \gamma
    \\ 
    \widehat{\Xi_{\gamma}}=0 \textit{ on  } \xi \in (2\gamma)^c
    \\
    \sum_{\gamma \in \Gamma_{\vec{\alpha}}(R^{-1}) }\widehat{\Xi_{\gamma}}(\xi) =1 \textit{ on  } \xi \in \mathcal{N}_{\mathbb{P}^{n-1}}(R^{-1}).
\end{gather}
We may further assume that $\|\Xi_{\gamma}\|_{L^1(\R^n)} \sim 1$.
We now bound \eqref{0629.411} by
\begin{equation}
    \lesssim
        R^{-(1-\frac1p)+2\epsilon}
    \Big( \sum_{\gamma \in \Gamma_{\vec{\alpha}}(R^{-1}) }\| \widehat{\psi_{B_R}}*(\widehat{\Xi_{\gamma} } \widehat{F })\|_{L^{p}(\R^n)}^p \Big)^{\frac1p}.
\end{equation}
We next apply the Hausdorff-Young inequality, and this is bounded by
\begin{equation}
    \begin{split}
        &\lesssim
        R^{-(1-\frac1p)+2\epsilon}
    \Big( \sum_{\gamma \in \Gamma_{\vec{\alpha}}(R^{-1}) }\| {\Xi_{\gamma}*F }\|_{L^{p'}(\psi_{B_R})}^p \Big)^{\frac1p}
    \\& \lesssim
     R^{-(1-\frac1p)+2\epsilon}
     R^{n(\frac{1}{p'}-\frac1p)}
     \Big( \sum_{\gamma \in \Gamma_{\vec{\alpha}}(R^{-1}) }\| {\Xi_{\gamma}*F }\|_{L^{p}(\psi_{B_R})}^p \Big)^{\frac1p}
     \\&\lesssim
     R^{-(1-\frac1p)+2\epsilon}
     R^{n(\frac{1}{p'}-\frac1p)}
     \Big( \sum_{\gamma \in \Gamma_{\vec{\alpha}}(R^{-1}) }\| {\Xi_{\gamma}*F }\|_{L^{p_c}(\psi_{B_R})}^{p_c} \Big)^{\frac{1}{p_c}}.
    \end{split}
\end{equation}
Note that
\begin{equation}
    R^{-(1-\frac1p)}
     R^{n(\frac{1}{p'}-\frac1p)} \lesssim_{\epsilon} R^{\epsilon} R^{|\vec{\alpha}|(\frac12-\frac1p)} \lesssim_{\epsilon} R^{2\epsilon} R^{|\vec{\alpha}|(\frac12-\frac{1}{p_c})}.
\end{equation}
So it suffices to prove
\begin{equation}\label{0629.416}
    \Big( \sum_{\gamma \in \Gamma_{\vec{\alpha}}(R^{-1}) }\| {\Xi_{\gamma}*F }\|_{L^{p_c}(\psi_{B_R})}^{p_c} \Big)^{\frac{1}{p_c}} \lesssim \Big( \sum_{\gamma \in \Gamma_{\vec{\alpha}}(R^{-1}) }\| {\mathcal{P}_{\gamma}F }\|_{L^{p_c}(\psi_{B_R})}^{p_c} \Big)^{\frac{1}{p_c}}.
\end{equation}
Once we have this inequality, the desired inequality simply follows by summing over balls $B_R$. To prove \eqref{0629.416},
we first note that
\begin{equation}
    \Xi_{\gamma}*F=\sum_{\substack{ \gamma' \in \Gamma_{\vec{\alpha}}(R^{-1}):\\|\gamma-\gamma'| \lesssim R^{-1/2}  }} \Xi_{\gamma}* \mathcal{P}_{\gamma'}F.
\end{equation}
By this observation, and H\"{o}lder's inequality with $\|\Xi_{\gamma}\|_{1} \lesssim 1$, and the triangle inequality, the left hand side of \eqref{0629.416} is bounded by
\begin{equation}
\begin{split}
    &\lesssim  \Big( \sum_{\gamma \in \Gamma_{\vec{\alpha}}(R^{-1}) }\Big\| \sum_{\substack{ \gamma' \in \Gamma_{\vec{\alpha}}(R^{-1}):\\|\gamma-\gamma'| \lesssim R^{-1/2}  }} \mathcal{P}_{\gamma'}F \Big\|_{L^{p_c}(\psi_{B_R})}^{p_c} \Big)^{\frac{1}{p_c}} 
    \\& \lesssim \Big( \sum_{\gamma \in \Gamma_{\vec{\alpha}}(R^{-1}) }\| {\mathcal{P}_{\gamma}F }\|_{L^{p_c}(\psi_{B_R})}^{p_c} \Big)^{\frac{1}{p_c}}.
   \end{split} 
\end{equation}
This completes the proof of \eqref{0629.416}.

\bibliographystyle{alpha}
\bibliography{reference}

\end{document}